\documentclass[11pt]{amsart}
\usepackage{graphicx,amscd,color}
\usepackage{a4wide}

\theoremstyle{plain}
\newtheorem{theorem}{Theorem}[section]
\newtheorem{proposition}[theorem]{Proposition}

\newtheorem{corollary}[theorem]{Corollary}

\newtheorem{question}[theorem]{Question}

\theoremstyle{remark}

\makeatletter
\@namedef{subjclassname@2020}{\textup{2020} Mathematics Subject Classification}
\makeatother

\begin{document}

\title [One-two-way pass-move]
{One-two-way pass-move for knots and links}

\author[H. Kim]{Hyejung Kim}
\address{Department of Mathematics and Institute of Pure and Applied Mathematics, Jeonbuk National University, Jeonju 54896, Korea}
\email{whsclsrn@jbnu.ac.kr}

\author[J. H. Lee]{Jung Hoon Lee}
\address{Department of Mathematics and Institute of Pure and Applied Mathematics, Jeonbuk National University, Jeonju 54896, Korea}
\email{junghoon@jbnu.ac.kr}

\subjclass[2020]{57K10}

\keywords{one-two-way pass-move, pass-move, $\#$-move}

\begin{abstract}
We define a local move for knots and links called the {\em one-two-way pass-move}, abbreviated briefly as the {\em $1$-$2$-move}.
The $1$-$2$-move is motivated from the pass-move and the $\#$-move, and it is a hybrid of them.
We show that the equivalence under the $1$-$2$-move for knots is the same as that of the pass-move:
a knot $K$ is $1$-$2$-move equivalent to an unknot (a trefoil respectively) if and only if
the Arf invariant of $K$ is $0$ ($1$ respectively).
On the other hand, we show that the number of $1$-$2$-moves behaves differently from the number of pass-moves.
\end{abstract}

\maketitle

\section{Introduction}\label{sec1}

There are various local moves on knot diagrams.
Each move either preserves or changes the knot type.
By studying knots combinatorially using local moves, we can obtain important results,
e.g. Reidemeister's theorem that any two diagrams of a knot can be changed to each other via a finite sequence of Reidemeister moves.

A local move is said to be an {\em unknotting operation} if any knot can be changed to an unknot via a finite sequence of the move.
The classical crossing change is an unknotting operation.
A $\#$-move is a kind of $2$-strand version of the classical crossing change for oriented diagrams, and
it is also an unknotting operation \cite{Murakami}.
When two parallel strands are over or under the other two parallel strands,
the {\em $\#$-move} changes the over/under information of all four crossings (Figure 1(b)).

A pass-move is similar to the $\#$-move, but it differs in the orientation.
When two anti-parallel strands are over or under the other two anti-parallel strands,
the {\em pass-move} changes the over/under information of all four crossings (Figure 1(a)).
The pass-move is not an unknotting operation, but
a knot $K$ is pass-move equivalent to an unknot (a trefoil respectively) if and only if
the Arf invariant of $K$ is $0$ ($1$ respectively) \cite{Kauffman}.

We define a new local move for knots and links, called the one-two-way pass-move,
which is a mixture of the pass-move and the $\#$-move.
When two parallel strands are over or under the other two anti-parallel strands,
the {\em one-two-way pass-move} changes the over/under information of all four crossings (Figure 1(c)).
We will simply call it a {\em $1$-$2$-move}.
The naming came from the one-way and two-way roads (Figure 2).

\begin{figure}[!hbt]
\includegraphics[width=12cm,clip]{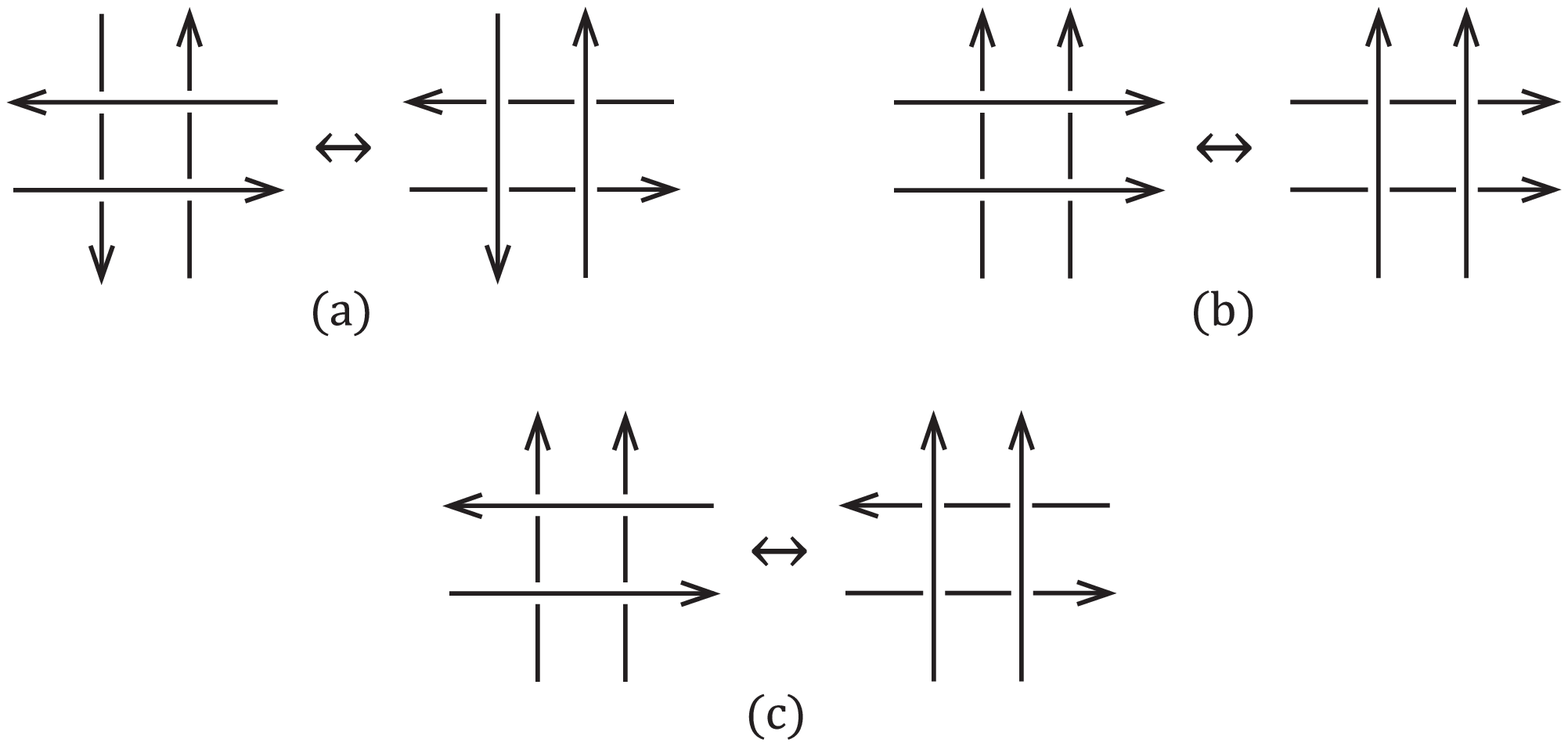}
\caption{(a) A pass-move, (b) a $\#$-move, and (c) a $1$-$2$-move.}\label{fig1}
\end{figure}

\begin{figure}[!hbt]
\includegraphics[width=5cm,clip]{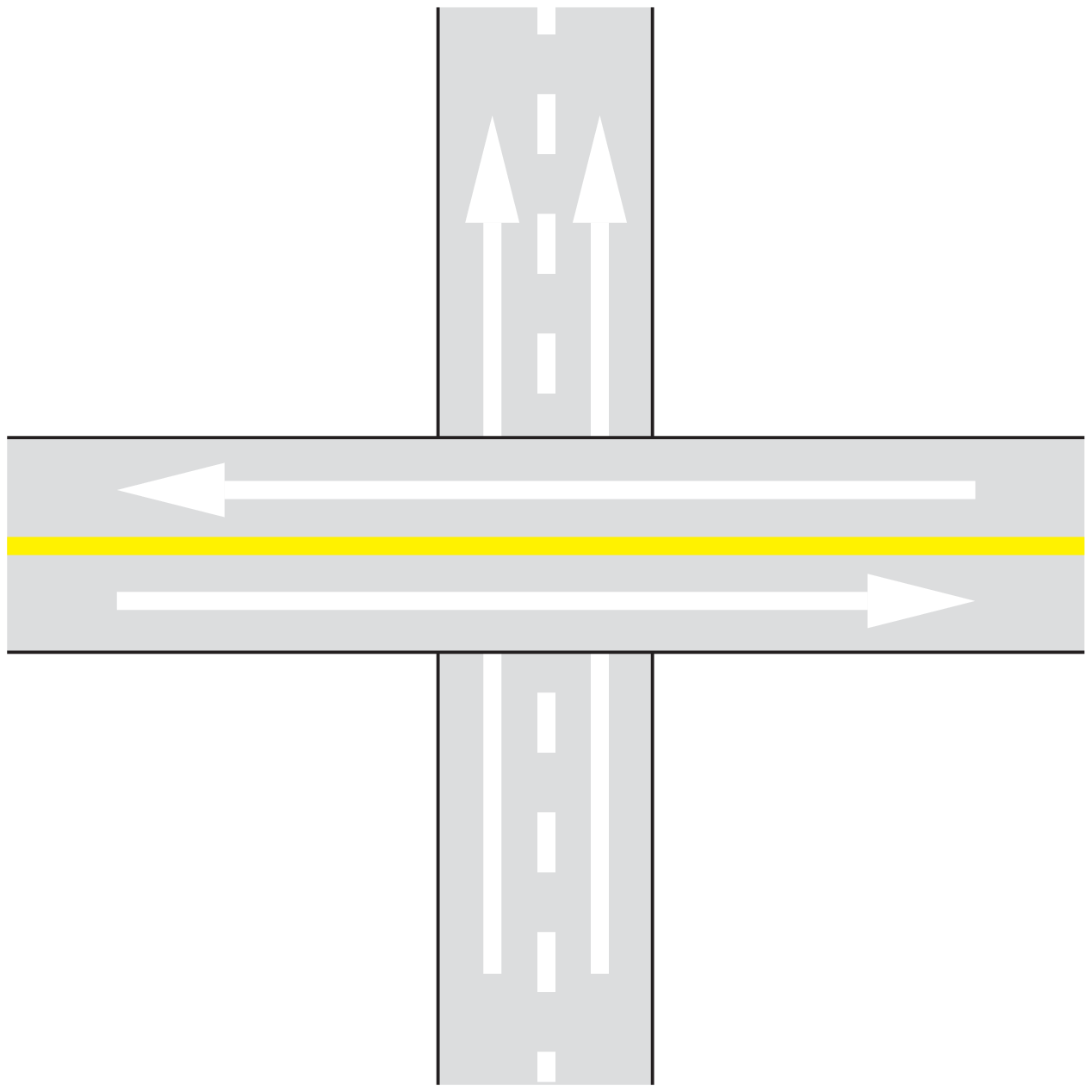}
\caption{One-way and two-way.}\label{fig2}
\end{figure}

We study basic relations between these three moves.
It turns out that the $1$-$2$-move has the same property as the pass-move.
For knots and proper links with Arf invariant $0$,
the {\em pass-move number} $p(L)$ and the {\em $1$-$2$-move number} $nt(L)$ can be defined
as the minimal number of moves required to be an unknot or an unlink.
(The definition of a proper link will be given in Section \ref{sec3}.)
We show that $p(L)$ and $nt(L)$ behave differently.
We give an example of a knot $K$ with $p(K) = nt(K) = 1$ and
another example of a link $L$ with $p(L) = 1$ and $nt(L) = 2$.
The knot $K$ is a composite knot, so it also serves as an example of a composite knot with $nt(K) = 1$.
We conjecture that there exists an example of a knot $K$ such that $p(K) = 1$ and $nt(K) = 2$.
More generally, we have the following question.

\begin{question}\label{question1.1}
Is there a knot $K$ such that $| nt(K) - p(K) |$ is a large number?
\end{question}

\section{Basic properties of the $1$-$2$-move}\label{sec2}

First we show a relation between the pass-move and the $1$-$2$-move.

\begin{proposition}\label{prop2.1}
A pass-move is realized by applying a $1$-$2$-move twice.
\end{proposition}

\begin{proof}
Without loss of generality, we assume that two horizontal anti-parallel strands are
over two vertical anti-parallel strands as in Figure 3.
A pass-move is achived by applying a RI-move, two RII-moves, two $1$-$2$-moves,
two RII-moves, and a RI-move, as illustrated in Figure \ref{fig3}.
(A RI-move (RII-move respectively) is a Reidemeister move I (II respectively).)

\begin{figure}[!hbt]
\includegraphics[width=11cm,clip]{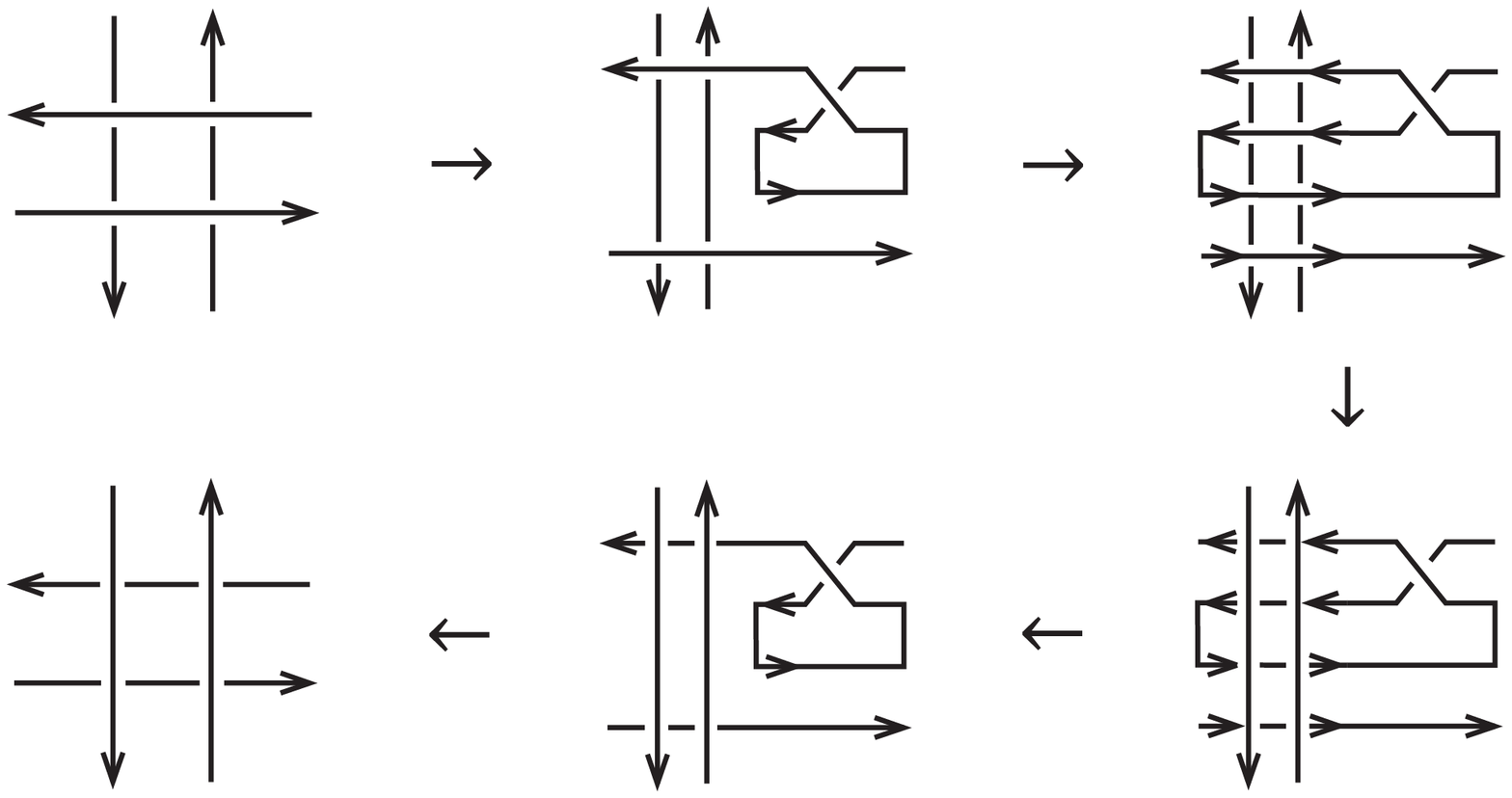}
\caption{A pass-move via two $1$-$2$-moves and Reidemeister moves.}\label{fig3}
\end{figure}
\end{proof}

As a corollary, we have the following.

\begin{corollary}\label{cor2.2}
$nt(L) \le 2 p(L)$.
\end{corollary}

Similarly, we have the following.

\begin{proposition}\label{prop2.3}
A $1$-$2$-move is realized by applying a $\#$-move twice.
\end{proposition}

\begin{proof}
Figure \ref{fig4} illustrates the relevant sequence of moves.

\begin{figure}[!hbt]
\includegraphics[width=11cm,clip]{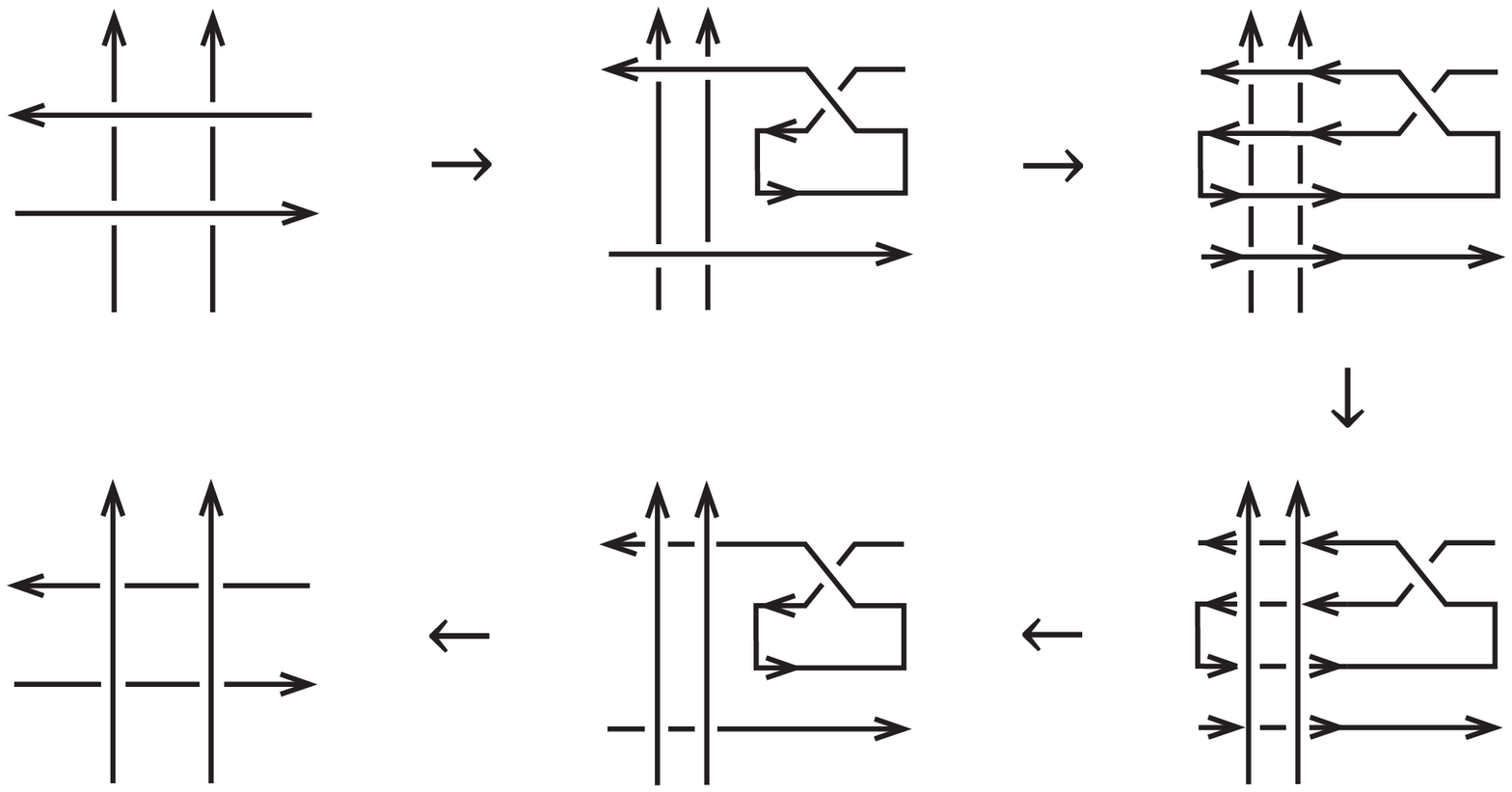}
\caption{A $1$-$2$-move via two $\#$-moves and Reidemeister moves.}\label{fig4}
\end{figure}
\end{proof}

So a pass-move is realized by applying a $\#$-move four times.

\section{Arf invariant}\label{sec3}

In this section, we briefly review some notions and properties related to the Arf invariant.
For more details, please refer to \cite{Kawauchi}.

Let $K_i$ denote a component of a link $L$.
A link $L$ is a {\em proper link} if the linking number $\mathrm{lk}(K_i, L - K_i) = 0 \pmod{2}$ for every $i$.
In particular, a knot is a proper link.

Let $b$ be a band embedded in $S^3$ with its left, right, top, bottom edges denoted by $b_l, b_r, b_t, b_b$ respectively.
Suppose that $b \cap L = b_l \cup b_r$, and
$b_l$ and $b_r$ have opposite orientations induced from $L$ and
belong to different components $K_i$ and $K_j$ of $L$ respectively.
Then we say that $L - (b_l \cup b_r) \cup (b_t \cup b_b)$ is a link obtained from $L$ by a {\em fusion}.
See Figure \ref{fig5}.

\begin{figure}[!hbt]
\includegraphics[width=8cm,clip]{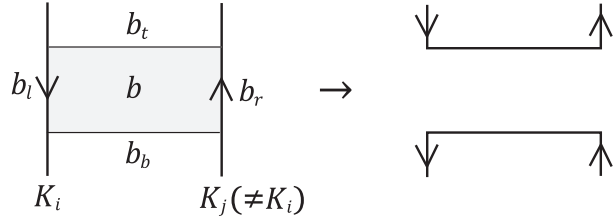}
\caption{A fusion.}\label{fig5}
\end{figure}

Let $F$ be a connected genus-$g$ Seifert surface of an $r$-component link $L$.
let $\mathcal{B} = \{ x_i, y_i, z_k \,|\, i=1, \ldots, g$ and $k=1, \ldots, r-1 \}$ be a basis of $H_1(F; \mathbb{Z}_2)$
represented by loops in $F$ such that $|x_i \cap y_j| = \delta_{ij}$ (the Kronecker delta) and
$z_k$ is a $k$-th component of $L$.
See Figure \ref{fig6}.

\begin{figure}[!hbt]
\includegraphics[width=9.5cm,clip]{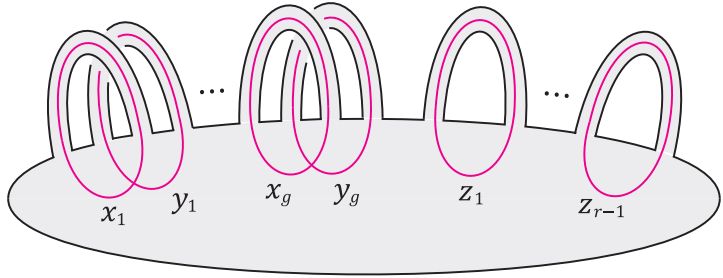}
\caption{A Seifert surface $F$.}\label{fig6}
\end{figure}

For a loop $l$ in $F$, let $q(l) = \mathrm{lk}(l^{+}, l) \pmod{2}$, where
$l^{+}$ is a loop obtained by slightly pushing $l$ to the positive direction of $F$.
For the basis $\mathcal{B}$, let
$$\mathrm{Arf}(F, \mathcal{B}) = \sum_{i=1}^{g} q(x_i)q(y_i) \pmod{2}.$$
This value is called the {\em Arf invariant} of $F$ with respect to $\mathcal{B}$.
An Arf invariant depends on the choice of $F$ and $\mathcal{B}$.
But for proper links, it is an invariant of a link.
The following are known results concerning the Arf invariant.

\begin{proposition}\label{prop3.1}
Suppose that $L_1$ and $L_2$ are proper links.
Then a split union $L_1 \sqcup L_2$ is also a proper link, and
$\mathrm{Arf}(L_1 \sqcup L_2) = \mathrm{Arf}(L_1) + \mathrm{Arf}(L_2) \pmod{2}$.
\end{proposition}

\begin{proposition}\label{prop3.2}
Suppose that $L$ is a proper link.
If $L'$ is a link obtained from $L$ by a fusion, then
$L'$ is also a proper link and $\mathrm{Arf}(L') = \mathrm{Arf}(L)$.
\end{proposition}

\begin{theorem}\label{thm3.3}
Suppose that $L_1$ and $L_2$ are proper links.
Then $L_1$ and $L_2$ are pass-move equivalent if and only if $\mathrm{Arf}(L_1) = \mathrm{Arf}(L_2)$.
\end{theorem}

We show that a result similar to Theorem \ref{thm3.3} holds for the $1$-$2$-move.

\begin{theorem}\label{thm3.4}
Suppose that $L_1$ and $L_2$ are proper links.
Then $L_1$ and $L_2$ are $1$-$2$-move equivalent if and only if $\mathrm{Arf}(L_1) = \mathrm{Arf}(L_2)$.
\end{theorem}

\begin{proof}
Suppose that $\mathrm{Arf}(L_1) = \mathrm{Arf}(L_2)$.
Then $L_1$ and $L_2$ are pass-move equivalent by Theorem \ref{thm3.3}.
Then by Proposition \ref{prop2.1}, $L_1$ and $L_2$ are $1$-$2$-move equivalent.

Conversely, suppose that $L_1$ and $L_2$ are $1$-$2$-move equivalent.
Let $L_0$ be an untwisted $2$-cable link of a Hopf link, where
one pair of two cables are parallel and the other pair of two cables are anti-parallel.
See Figure \ref{fig7}(a).
The link $L_0$ is a proper link.
Since we can obtain an unlink by banding the two anti-parallel components of $L_0$ and
$\mathrm{Arf}$(an unlink$) = 0$,
$\mathrm{Arf}(L_0) = 0$ by Proposition \ref{prop3.2}.
By Proposition \ref{prop3.1}, $L_1 \sqcup L_0$ is a proper link and
$\mathrm{Arf}(L_1 \sqcup L_0) = \mathrm{Arf}(L_1)$.
Performing a fusion operation four times to $L_1 \sqcup L_0$ has the same effect as a $1$-$2$-move on $L_1$ (Figure \ref{fig7}).
Let $L'_1$ be a link obtained from $L_1$ by a single $1$-$2$-move.
Then $\mathrm{Arf}(L'_1) = \mathrm{Arf}(L_1 \sqcup L_0) = \mathrm{Arf}(L_1)$.
By applying the above argument finitely many times, we conclude that $\mathrm{Arf}(L_1) = \mathrm{Arf}(L_2)$.

\begin{figure}[!hbt]
\includegraphics[width=12.5cm,clip]{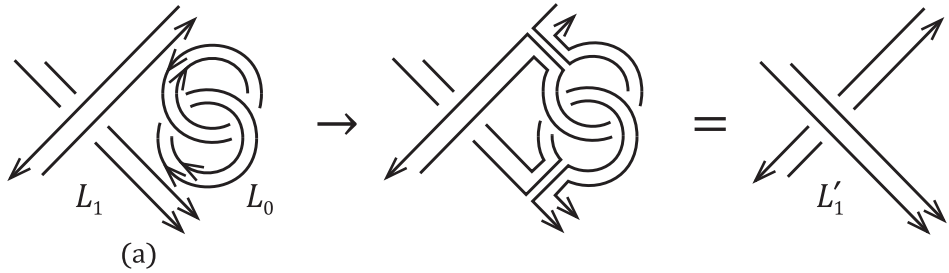}
\caption{A $1$-$2$-move via four fusions.}\label{fig7}
\end{figure}
\end{proof}

\section{A knot $K$ with $p(K) = nt(K) = 1$}\label{sec4}

Let $K =$ (a left-hand trefoil) $\#$ (a right-hand trefoil).

\begin{figure}[!hbt]
\includegraphics[width=10cm,clip]{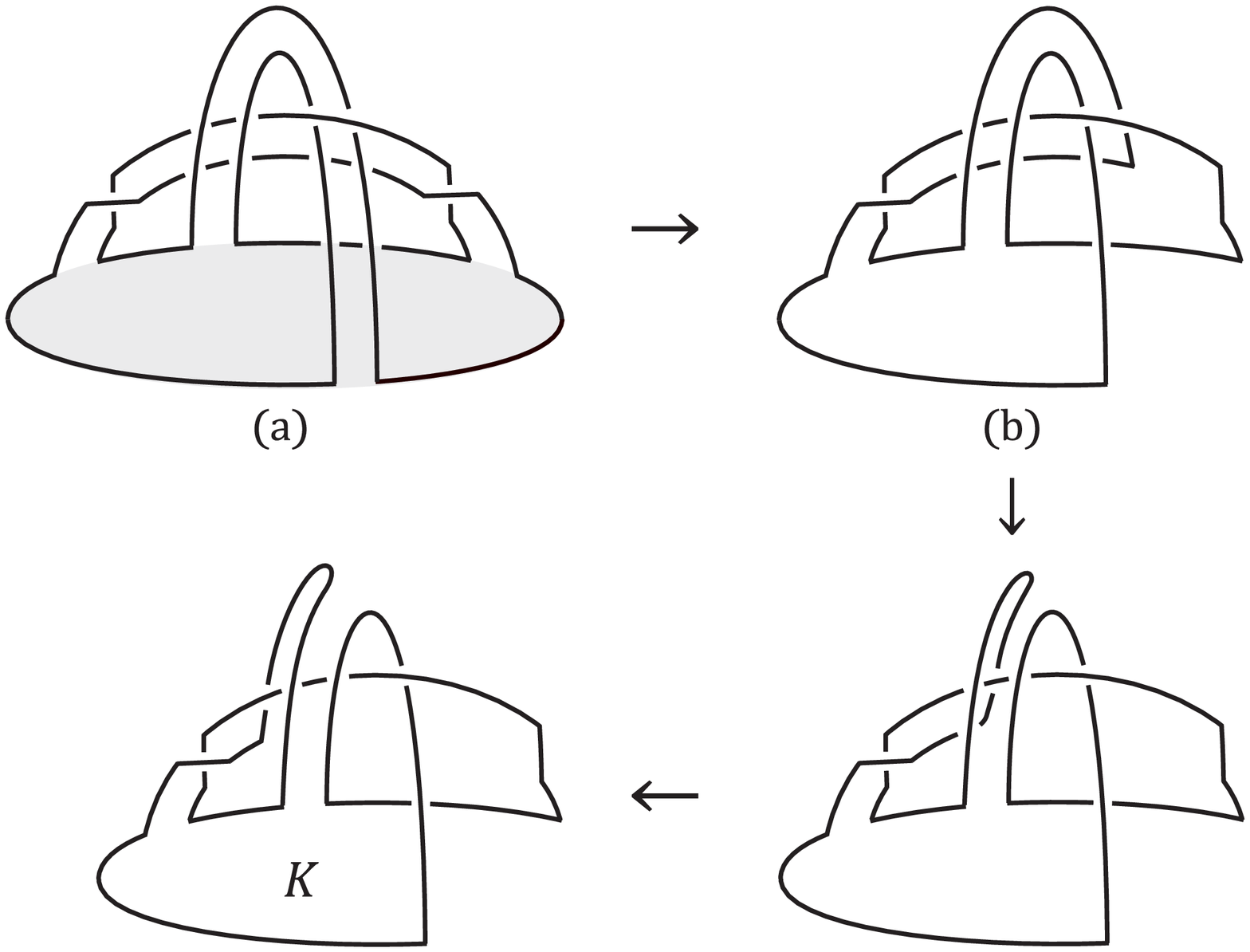}
\caption{$K =$ (a left-hand trefoil) $\#$ (a right-hand trefoil).}\label{fig8}
\end{figure}

\begin{theorem}\label{thm4.1}
$p(K) = 1$ and $nt(K) = 1$.
\end{theorem}

\begin{proof}
The knot in Figure \ref{fig8}(a) is obtained from a disk by two banding operations and taking a boundary knot.
It is isotoped to $K$ as illustrated in Figure \ref{fig8}.
It is well known that $p(K) = 1$ as shown in Figure \ref{fig9}.
Figure \ref{fig10}(a) and Figure \ref{fig8}(b) are the same, and Figure \ref{fig10} shows that $nt(K) = 1$.

\begin{figure}[!hbt]
\includegraphics[width=13cm,clip]{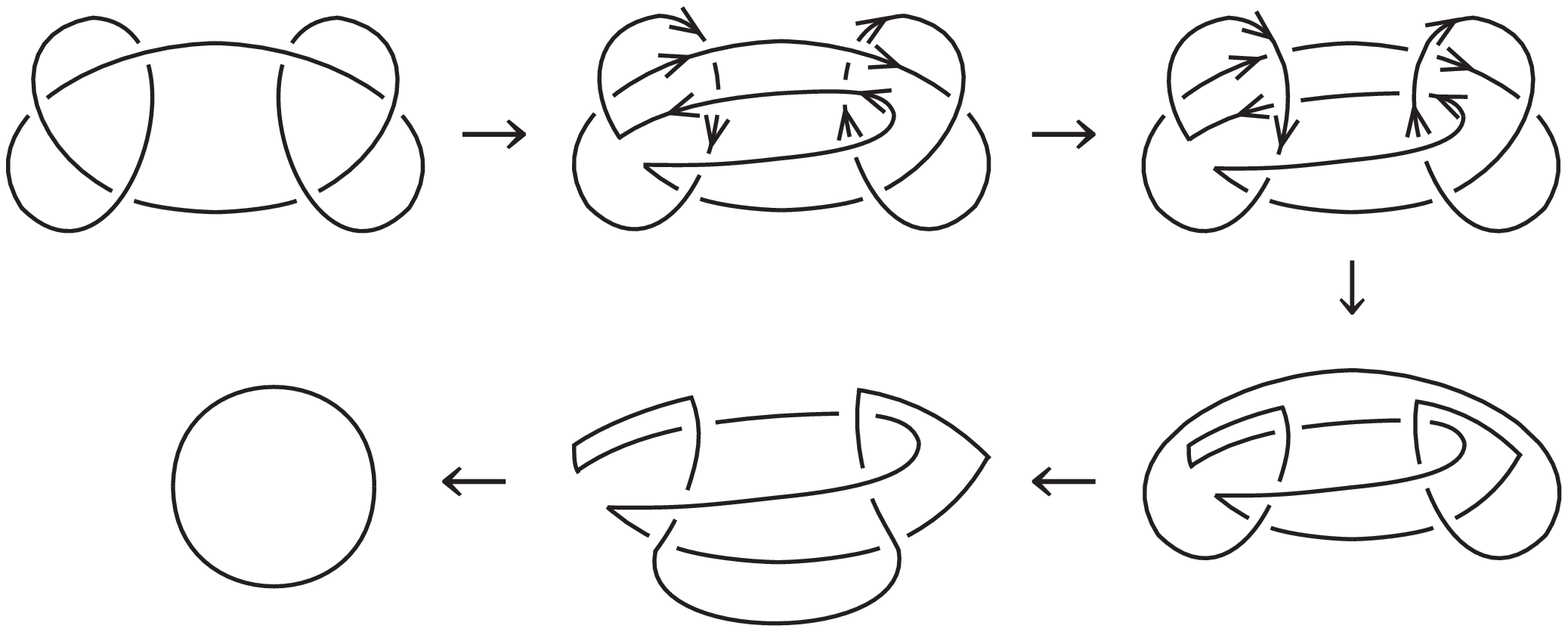}
\caption{$p(K) = 1$.}\label{fig9}
\end{figure}

\begin{figure}[!hbt]
\includegraphics[width=10.5cm,clip]{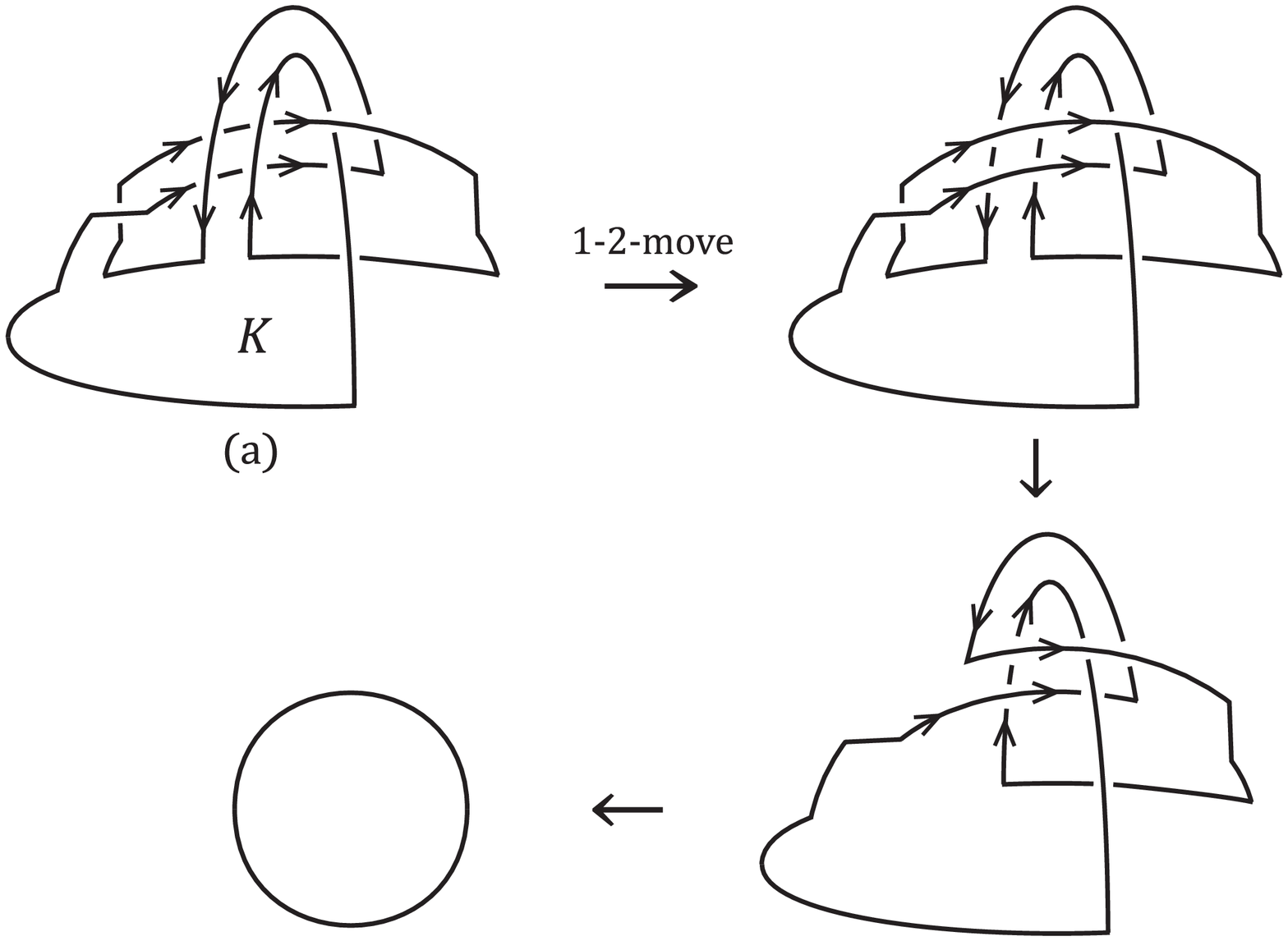}
\caption{$nt(K) = 1$.}\label{fig10}
\end{figure}
\end{proof}

\section{A link $L$ with $p(L) = 1$ and $nt(L) = 2$}\label{sec5}

Let $L$ be an untwisted $2$-cable of a Hopf link, where the two cables have opposite orientations for each pair.
See Figure \ref{fig11}.
It is obvious that $p(L)=1$.
We show that $nt(L)=2$.
Therefore the $1$-$2$-move number is not the same as the pass-move number.

We label the four components of $L$ and an unlink by $a, b, c, d$ respectively as in Figure \ref{fig11}.
Note that for $L$, the linking numbers satisfy
$\mathrm{lk}(a,c) = -1$, $\mathrm{lk}(a,d) = 1$, $\mathrm{lk}(b,c) = 1$, $\mathrm{lk}(b,d) = -1$,
$\mathrm{lk}(a,b) = \mathrm{lk}(c,d) = 0$,
whereas for the unlink,
$\mathrm{lk}(a,c) = \mathrm{lk}(a,d) = \mathrm{lk}(b,c) = \mathrm{lk}(b,d) = \mathrm{lk}(a,b) = \mathrm{lk}(c,d) = 0$.

\begin{figure}[!hbt]
\includegraphics[width=10cm,clip]{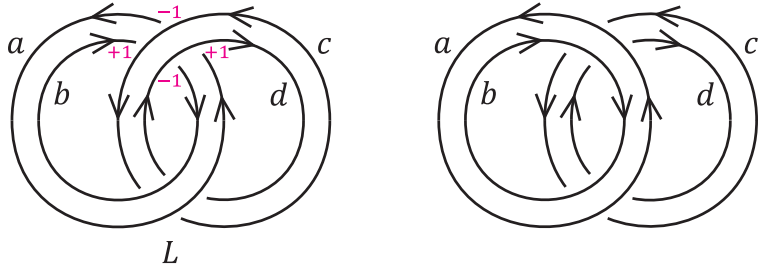}
\caption{The link $L$ and an unlink.}\label{fig11}
\end{figure}

\begin{theorem}\label{thm5.1}
$nt(L) = 2$
\end{theorem}

\begin{proof}
Suppose that $nt(L) \ne 2$.
Since $nt(L) \le 2 p(L) = 2$ by Corollary \ref{cor2.2}, $nt(L) = 1$.
Consider a diagram $D$ of $L$ such that a single $1$-$2$-move on $D$ yields a diagram $D_0$ of an unlink.
Since by only a single $1$-$2$-move all linking numbers
$\mathrm{lk}(a,c)$, $\mathrm{lk}(a,d)$, $\mathrm{lk}(b,c)$, $\mathrm{lk}(b,d)$ change to $0$,
the four components $a, b, c, d$ should be involved in the four strands of the $1$-$2$-move on $D$.
There are two cases.

Case 1. Two parallel strands are over two anti-parallel strands. \\
We label the four strands of $D$ and $D_0$ by $1, 2, 3, 4$ respectively.
Without loss of generality, we assume that $D$ and $D_0$ are as in Figure \ref{fig12}.

\begin{figure}[!hbt]
\includegraphics[width=7cm,clip]{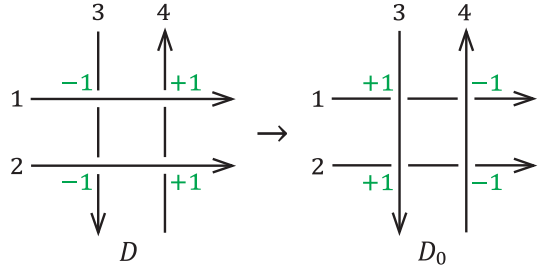}
\caption{Two parallel strands are over two anti-parallel strands..}\label{fig12}
\end{figure}

Case 1.1. The component $a$ belongs to the parallel strands. \\
Without loss of generality, we assume that $a = 1$. See Figure \ref{fig13}.
Then $b = 2$.
Since $\mathrm{lk}(a, c)$ increases by one from $-1$ to $0$, the component $c$ should be component $3$.
Then $\mathrm{lk}(b, c)$ also increases by one after the $1$-$2$-move and
it contradicts that $\mathrm{lk}(b, c) = 1$ in $L$ and $\mathrm{lk}(b, c) = 0$ in the unlink.

\begin{figure}[!hbt]
\includegraphics[width=7cm,clip]{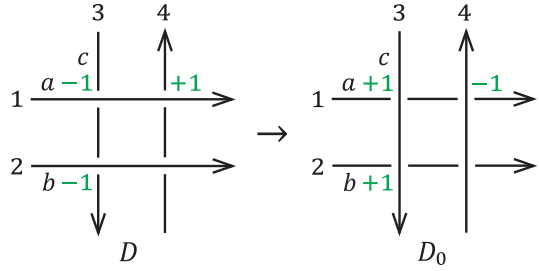}
\caption{The component $a$ belongs to the parallel strands.}\label{fig13}
\end{figure}

Case 1.2. The component $a$ belongs to the anti-parallel strands. \\
Without loss of generality, we assume that $a = 3$. See Figure \ref{fig14}.
Then $c = 1$ and $d = 2$, or $c = 2$ and $d = 1$.
In any case, both $\mathrm{lk}(a, c)$ and $\mathrm{lk}(a, d)$ increase by one after the $1$-$2$-move.
It contradicts that $\mathrm{lk}(a, c) = -1$ and $\mathrm{lk}(a, d) = 1$ in $L$ and
$\mathrm{lk}(a, c) = \mathrm{lk}(a, d) = 0$ in the unlink.

\begin{figure}[!hbt]
\includegraphics[width=7cm,clip]{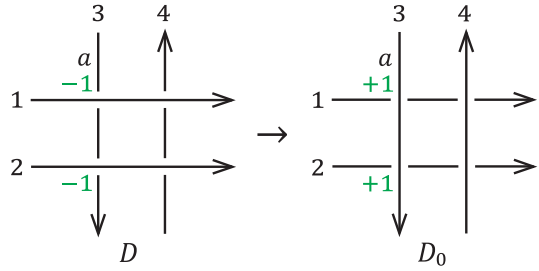}
\caption{The component $a$ belongs to the anti-parallel strands.}\label{fig14}
\end{figure}

Case 2. Two parallel strands are under two anti-parallel strands. \\
As above, we label the four strands of $D$ and $D_0$ by $1, 2, 3, 4$ respectively.
Without loss of generality, we assume that $D$ and $D_0$ are as in Figure \ref{fig15}.

\begin{figure}[!hbt]
\includegraphics[width=7cm,clip]{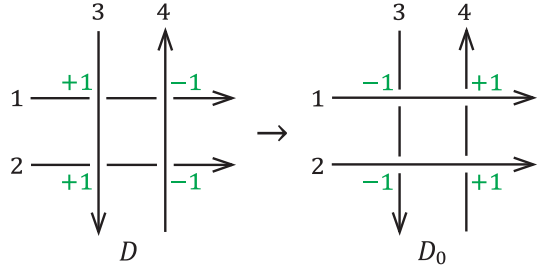}
\caption{Two parallel strands are under two anti-parallel strands.}\label{fig15}
\end{figure}

Case 2.1. The component $a$ belongs to the parallel strands. \\
Without loss of generality, we assume that $a = 1$. See Figure \ref{fig16}.
Then $b = 2$.
Since $\mathrm{lk}(a, c)$ increases by one from $-1$ to $0$, the component $c$ should be component $4$.
Then $\mathrm{lk}(b, c)$ also increases by one after the $1$-$2$-move and
it contradicts that $\mathrm{lk}(b, c) = 1$ in $L$ and $\mathrm{lk}(b, c) = 0$ in the unlink.

\begin{figure}[!hbt]
\includegraphics[width=7cm,clip]{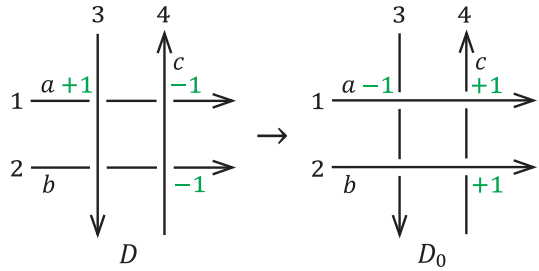}
\caption{The component $a$ belongs to the parallel strands.}\label{fig16}
\end{figure}

Case 2.2. The component $a$ belongs to the anti-parallel strands. \\
Without loss of generality, we assume that $a = 3$. See Figure \ref{fig17}.
Then $c = 1$ and $d = 2$, or $c = 2$ and $d = 1$.
In any case, both $\mathrm{lk}(a, c)$ and $\mathrm{lk}(a, d)$ decrease by one after the $1$-$2$-move.
It contradicts that $\mathrm{lk}(a, c) = -1$ and $\mathrm{lk}(a, d) = 1$ in $L$ and
$\mathrm{lk}(a, c) = \mathrm{lk}(a, d) = 0$ in the unlink.

\begin{figure}[!hbt]
\includegraphics[width=7cm,clip]{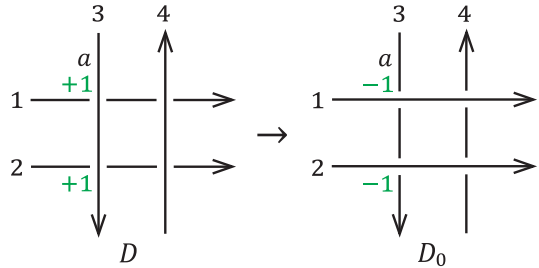}
\caption{The component $a$ belongs to the anti-parallel strands.}\label{fig17}
\end{figure}

\end{proof}




\begin{thebibliography}{00}

\bibitem{Kauffman} L. H. Kauffman,
Formal knot theory,
Mathematical Notes, 30. Princeton University Press, Princeton, NJ, 1983. ii+168 pp. ISBN: 0-691-08336-3.

\bibitem{Kawauchi} A. Kawauchi,
A survey of knot theory,
Birkh\"{a}user Verlag, Basel, 1996.
Translated and revised from the 1990 Japanese original by the author.

\bibitem{Murakami} H. Murakami,
Some metrics on classical knots,
Math. Ann. {\bf 270} (1985), no. 1, 35--45.




\end{thebibliography}
\end{document}